\documentclass[11pt,hyp,reqno,]{amsart}
\usepackage{hyperref}
\usepackage{enumerate}
\hypersetup{nesting=true,debug=true,naturalnames=true}
\usepackage{graphicx,amssymb,upref}
\usepackage[utf8]{inputenc}

\usepackage[backend=bibtex,
style=alphabetic,
bibencoding=ascii,
sorting=ynt
]{biblatex}
\renewbibmacro{in:}{}
\usepackage{fancyhdr}
\fancypagestyle{firststyle}
{
    \setlength{\headheight}{25pt}
    \setlength{\headsep}{15pt}
}

\let\<\langle
\let\>\rangle
\usepackage[all,pdf]{xy}
\UseComputerModernTips

\let\uml\"

\title{Faithfulness of the Fock representation of $C^*$-algebra generated by $q_{ij}$-commuting isometries} 

\author{Alexey Kuzmin and Nikolay Pochekai}  
\email{vagnard.k@gmail.com}
\newcommand{\norm}[1]{\left\lVert#1\right\rVert}

\subjclass[2010]{}

\newtheorem{theorem}{Theorem}
\newtheorem{lemma}{Lemma}

\newtheorem{remark}{Remark}
\newtheorem{proposition}{Proposition}
\newtheorem{definition}{Definition}
\newtheorem{corollary}{Corollary}

\begin{document} 

\maketitle

\begin{abstract}
   We consider $C^*$-algebra $Isom_{q_{ij}}$ generated by $n$ isometries $a_1, \ldots, a_n$ satisfying the relations $a_i^* a_j = q_{ij} a_j a_i^*$ with $\max |q_{ij}| < 1$. This $C^*$-algebra is shown to be nuclear. We prove that the Fock representation of $Isom_{q_{ij}}$ is faithful. Further we describe an ideal in $Isom_{q_{ij}}$ which is isomorphic to the algebra of compact operators.
\end{abstract}

\tableofcontents

%%%% **** The text of the paper starts here **** %%%%

\section{Introduction}
Operator algebras generated by various deformations of the canonical commutation relations (CCR) have been extensively studied in the last two decades. In particular, a considerable attention has been paid to the study of so-called $q_{ij}$-CCR introduced by M. Bozejko and R. Speicher, see \cite{Bozejko_brownian_motion}. Namely, $q_{ij}$-CCR is a $*$-algebra generated by $a_i, a_i^*$, $i = 1, \ldots, n$, which satisfy the following relations
\[ a_i^* a_i = 1 + q_{ii} a_i a_i^*, \text{ } |q_{ii}| < 1, \text{ } q_{ii} \in \mathbb{R}, \]
\[ a_i^* a_j = q_{ij} a_j a_i^* \text{ for } i \neq j, \text{ } |q_{ij}| \leq 1, \text{ } \overline{q_{ji}} = q_{ij}. \]
It is a deformation of $*$-algebras of the classical commutation relations in the sense that in the Fock realisation the limiting cases $q_{ij} = 1$ and $q_{ij} = -1$ correspond to $*$-algebras of the canonical commutation
relations (CCR) and the canonical anti-commutation relations (CAR) respectively.

Let us describe the main results on the subject:

The Fock representation $\pi_F$ of $q_{ij}$-CCR is the special one, determined uniquely up to a unitary equivalence by the following property: there exists a cyclic vector $\Omega$ such that $ \pi_F(a_i^*)(\Omega) = 0$ for $i = 1, \ldots, n $. The problem of existence and uniqueness of $\pi_F$ was studied in \cite{Fivel_positivity}, \cite{Bozejko_brownian_motion}, \cite{Zagier_positivity} and in \cite{Jorgensen_positivity} for a more general class of Wick algebras.
    
It can be easily verified that in any $*$-representation $\pi$ of $q_{ij}$-$CCR$ by bounded operators one has
\[ \norm{\pi(a_i)} \leq \frac{1}{\sqrt{1-|q_{ii}|}}, \text{ } i = 1, \ldots, n. \]
Hence there exists a universal enveloping $C^*$-algebra associated to $q_{ij}$-CCR denoted below by $q_{ij}$-$\mathcal{CCR}$.

%%%header problem fix
\thispagestyle{firststyle}

When $\max |q_{ij}| < 1$ it is natural to think of $q_{ij}$-$\mathcal{CCR}$ as a deformation of the Cuntz-Toeplitz algebra $\mathcal{KO}_n$. Recall that $\mathcal{KO}_n$ is the universal $C^*$-algebra generated by $s_i, s_i^*$, $i = 1, \ldots, n$ subject to the relations \[ s_i^* s_j = \delta_{ij}1, \text{ } i, j = 1,\ldots, n.\] This point of view was justified by P. E. T. Jorgensen, L. M. Schmidt and R. F. Werner who showed that $q_{ij}$-$\mathcal{CCR} \simeq \mathcal{KO}_n$ whenever $\max |q_{ij}| < \sqrt{2} - 1$, see \cite{Jorgensen_stability} for more details. It is a conjecture that $q_{ij}$-$\mathcal{CCR} \simeq \mathcal{KO}_n$ whenever $\max |q_{ij}| < 1$.

Many authors were also interested in the study of $C^*$-algebra generated by operators of the Fock representation of $q_{ij}$-$CCR$. Namely, K. Dykema and A. Nica in \cite{Dykema_Fock} proved that $\pi_F(q_{ij}$-$\mathcal{CCR})$ $ \simeq \mathcal{KO}_n$ for slightly larger value of $\max |q_{ij}|$. Also an embedding of $\mathcal{KO}_n$ into $\pi_F(q_{ij}$-$\mathcal{CCR})$ was constructed for any values of deformation parameters with $\max |q_{ij}| < 1$.

Later M. Kennedy in \cite{Kennedy_exactness} showed the existence of an embedding of $\pi_F(q_{ij}$-$\mathcal{CCR})$ into $\mathcal{KO}_n$ and proved that $\pi_F(q_{ij}$-$\mathcal{CCR})$ is an exact $C^*$-algebra.

Let us stress out that results concerning $\pi_F(q_{ij}$-$\mathcal{CCR})$ cannot be automatically lifted to the universal $C^*$-algebra level since at the moment we don't know whether or not $\pi_F$ is a faithful $*$-representation of $q_{ij}$-$\mathcal{CCR}$ for any $|q_{ij}| < 1$. However $\pi_F$ is a faithful representation of $q_{ij}$-$CCR$, i.e. it is faithful on the $*$-algebraic level, see \cite{Proskur_faithfulness}.

Some boundary cases of $q_{ij}$-$\mathcal{CCR}$ corresponding to $|q_{ii}| < 1$, $i = 1,\ldots,n$, $|q_{ij}| = 1$, $i \neq j$ were studied in \cite{Proskur_stability2}, \cite{Proskur_Iksanov}, see also \cite{Weber}. For these values of $q_{ij}$ it was shown that $q_{ij}$-$\mathcal{CCR} \simeq Isom_{q_{ij}}$ where $Isom_{q_{ij}}$ is a $*$-algebra generated by isometries $a_1, \ldots, a_n$ such that $a_i^* a_j = q_{ij} a_j a_i^*$, $i \neq j$. It was also shown that for the values of parameters $q_{ij}$ specified above, the $C^*$-algebra $Isom_{q_{ij}}$ is nuclear and its Fock representation is faithful.
    
Notice, that the $C^*$-algebras $Isom_{q_{ij}}$ with $\max |q_{ij}| < 1$ have completely different structure than in the case of $|q_{ij}| = 1$, $i \neq j$. It was shown in \cite{Proskur_stability2} that for two generators and $|q_1| = |q_2| = 1$, $Isom_{q_1} \simeq Isom_{q_2}$ if and only if $q_1 = q_2$ or $q_1 = \overline{q_2}$. If $|q| < 1$ one has $Isom_q \simeq \mathcal{KO}_2$ as established in \cite{Proskurin_q_isom}. For more than two generators the problem of isomorphism of $Isom_{q_{ij}}$ and $\mathcal{KO}_n$ with $\max |q_{ij}| < 1$ is still open.

In this paper we study the $C^*$-algebras $Isom_{q_{ij}}$ with $\max|q_{ij}| < 1$, $i, j = 1, \ldots, n$.

When a $C^*$-algebra admits an action of a compact group, the study of some of its properties can be reduced to the level of fixed point subalgebra by this action. In Section 3 we discuss conditions for a compact group to define a filtration preserving action on $q_{ij}$-$\mathcal{CCR}$ (and in particular on $Isom_{q_{ij}}$). We find the largest admissible group acting on $q_{ij}$-$\mathcal{CCR}$ regardless of the values of $q_{ij}$ - it happens to be the n-dimensional torus $\mathbb{T}^n$.

In Section 4 we begin to study the fixed point subalgebra of $Isom_{q_{ij}}$ with respect to the action of $\mathbb{T}^n$ named $\mathcal{GICAR}_{q_{ij}}$. This subalgebra turns out to be an AF algebra. As a consequence it will follow that $Isom_{q_{ij}}$ is nuclear.

In Section 5 we describe the Bratteli diagram of $\mathcal{GICAR}_{q_{ij}}$ and conclude its independence of the values of $q_{ij}$. Hence, we get an isomorphism of the fixed point subalgebras of $Isom_{q_{ij}}$ and $\mathcal{KO}_n$.

Another problem which can be reduced to the fixed point subalgebra level is faithfulness of a $*$-homomorphism. Using information on structure of the fixed point subalgebra, we prove in Section 6 that the Fock representation of $Isom_{q_{ij}}$, $\max |q_{ij}| < 1$, is faithful. This allows us to extend results and techniques of \cite{Dykema_Fock}, \cite{Kennedy_exactness} to the case of $Isom_{q_{ij}}$.

In Section 7 we prove the existence of an ideal $\mathcal{K} \subset Isom_{q_{ij}}$ isomorphic to the algebra of compact operators and describe a generator of this ideal as a projection in some finite-dimensional subalgebra of $Isom_{q_{ij}}$.

\section{Deformed Fock inner product}
As it was mentioned in the introduction, the $q_{ij}$-deformed Fock representation was a subject of numerous studies. For our purpose we only need a few basic facts about the structure of the Fock representation space of $q_{ij}$-$CCR$.

Let $\mathcal{H} = \mathbb{C}^n$ and $\{\xi_1,\ldots,\xi_n \}$ be its orthonormal basis. Consider the full tensor space
\[ T(\mathcal{H}) = \{ \Omega \} \oplus \bigoplus_{k \geq 1} \mathcal{H}^{\otimes k}. \]
Put $\mathcal{H}^{\otimes k} \perp \mathcal{H}^{\otimes l}$, $k \neq l$, and supply each $\mathcal{H}^{\otimes k}$ with inner product $\langle \cdot, \cdot \rangle_{Fock}$ specified below, see \cite{Bozejko_brownian_motion} for more details. Namely,

\begin{gather*}
    \langle \Omega, \Omega \rangle_{Fock} = 1, \\
    \langle \xi_{i_1} \otimes \ldots \otimes \xi_{i_k}, \xi_{j_1} \otimes \ldots \otimes \xi_{j_k} \rangle_{Fock} = \\ = \sum_{t = 1}^k \delta_{j_1 i_t} q_{j_1 i_1} \ldots q_{j_1 i_{t - 1}} \langle \xi_{i_1} \otimes \ldots \otimes \widehat{\xi_{i_t}} \otimes \ldots \otimes \xi_{i_k}, \xi_{j_2} \otimes \ldots \otimes \xi_{j_k} \rangle_{Fock},
\end{gather*}
where $1 \leq i_1, \ldots, i_k, j_1, \ldots, j_k \leq n$, and the hat over $\xi_{i_t}$ means that $\xi_{i_t}$ is deleted from the
tensor. Note that the natural basis $\xi_{i_1} \otimes \ldots \otimes \xi_{i_k}$ of $\mathcal{H}^{\otimes k}$ is not orthogonal with respect to $\langle \cdot, \cdot \rangle_{Fock}$.

\section{Compact group actions on $q_{ij}$-$\mathcal{CCR}$}

In this part of our paper we discuss symmetries of $q_{ij}$-$\mathcal{CCR}$ and explain how can they be useful in the study of this $C^*$-algebra. First, recall a definition of a group action.

\begin{definition}
Let $G$ be a compact group, $A$ be a $C^*$-algebra.
\begin{enumerate}
    \item An action of $G$ on $A$ is a homomorphism $\gamma : G \rightarrow Aut(A)$ which is continuous in the point-norm topology.
    \item The fixed point subalgebra $A^\gamma$ is a subset of all $a \in A$ such that $\gamma_g(a) = a$ for all $g \in G$.
\end{enumerate}
\end{definition}

Recall that for every action of a compact group $G$ on a $C^*$-algebra $A$ one can construct a faithful conditional expectation $E_{\gamma} : A \rightarrow A^\gamma$ onto the fixed point subalgebra, given by
\[ E_{\gamma}(a) = \int_{G} \gamma_g(a) d \lambda, \]
where $\lambda$ is the Haar measure on $G$.

The following proposition explains our interest in studying the fixed point subalgebras.

\begin{proposition}
    
(\cite{Ozawa} Section 4.5, Theorem 1, 2)

\begin{enumerate}
    \item Let $\gamma$ be an action of a compact group $G$ on a $C^*$-algebra $A$. Then $A$ is nuclear if and only if $A^\gamma$ is nuclear.
    \item Let $\alpha, \beta$ be actions of a compact group $G$ on $C^*$-algebras $A$ and $B$ respectively and $\pi : A \rightarrow B$ is a $*$-homomorphism such that 
    \[ \pi \circ \alpha_g = \beta_g \circ \pi \text{ for any } g \in G.\] Then $\pi$ is injective on $A$ if and only if $\pi$ is injective on $A^\alpha$.
\end{enumerate}
\end{proposition}

Our aim is to find a way to generate appropriate actions on $q_{ij}$-$\mathcal{CCR}$. Recall that Cuntz-Toeplitz algebra $\mathcal{KO}_n$ admits natural filtration preserving action for every closed subgroup of the unitary group $U_n$. For the readers convenience we put below a short prove of this well-known fact.

\begin{proposition}
Let $G$ be a closed subgroup of $U_n$ and $s_i$, $i = 1,\ldots,n$ are generators of $\mathcal{KO}_n$. Then there exists an action $\gamma$ of $G$ on $\mathcal{KO}_n$ such that
\[\gamma_g(s_i) = \sum_{j = 1}^n s_j g_{ji}, \text{ } g = (g_{ij}) \in G. \]

\end{proposition}
\begin{proof}
One has to examine that $\gamma_g(s_i)$, $i = 1,\ldots,n$, satisfy the basic relations of $\mathcal{KO}_n$. Indeed,
\[ \gamma_g(s_i^*) \gamma_g(s_j) = \sum_{k, l = 1}^n s_k^*s_l \overline{g_{ki}} g_{lj} = \sum_{k, l = 1}^n \delta_{kl} \overline{g_{ki}} g_{lj} = \sum_{k = 1}^n \overline{g_{ki}} g_{kj} = \delta_{ij}. \]

Thus, there is a correctly defined $*$-homomorphism $\gamma_g : \mathcal{KO}_n \rightarrow \mathcal{KO}_n$. Since $g$ is unitary, this $*$-homomorphism is an automorphism.

\end{proof}

Our goal is to verify when the construction above is applicable to  $q_{ij}$-$\mathcal{CCR}$. Unlike $\mathcal{KO}_n$, not every closed subgroup of $U_n$ defines a correct action on $q_{ij}$-$\mathcal{CCR}$. In the next lemma we state the conditions for an element of $U_n$ to define an automorphism of $q_{ij}$-$\mathcal{CCR}$.

\begin{lemma}
Suppose $u \in U_n$, $a_i, i = 1,\ldots,n$ are generators of $q_{ij}$-$\mathcal{CCR}$ and
\[ a_i' = \sum_{j = 1}^n a_j u_{ji}, \text{ } i = 1,\ldots,n. \]
Then the relations 
\[ a_i'^*a_j' - q_{ij} a_j' a_i'^* = \delta_{ij}, \text{ } i, j = 1, \ldots, n\] hold if and only if 
\[ q_{kl}\overline{u_{ki}} u_{lj} = q_{ij} u_{lj} \overline{u_{ki}}, \text{ } i, j, k, l = 1, \ldots, n. \]
\end{lemma}
\begin{proof}
Suppose that 
\[ a_i'^*a_j' - q_{ij} a_j' a_i'^* = \delta_{ij}, \text{ for any } i, j = 1, \ldots, n.\]
Then
\begin{equation*}
\begin{split}
\delta_{ij} = a_i'^*a_j' - q_{ij} a_j' a_i'^* & = \sum_{k, l} (a_k^* a_l \overline{u_{ki}} u_{lj} - q_{ij} a_l a_k^* u_{lj} \overline{u_{ki}})  \\
 & = \sum_{k, l} ((\delta_{lk} 1 + q_{kl} a_l a_k^*) \overline{u_{ki}} u_{lj} - q_{ij} a_l a_k^* u_{lj} \overline{u_{ki}} )
\\
 & = \sum_{k} \overline{u_{ki}} u_{kj} + \sum_{k, l} (q_{kl} a_l a_k^* \overline{u_{ki}} u_{lj} - q_{ij} a_l a_k^* \overline{u_{ki}} u_{lj})
\\
 & = \delta_{ij} + \sum_{k, l} a_l a_k^* \overline{u_{ki}} u_{lj} (q_{kl} - q_{ij}).
\end{split}
\end{equation*}
The linear independency of $a_l a_k^*$, $l,k = 1,\ldots,n$, implies the claim.

\end{proof}

Notice that if groups $G_1, G_2 \subset U_n$ act on $q_{ij}$-$\mathcal{CCR}$ as in Lemma 1 and $G_1 \leq G_2$ then the fixed point subalgebra by the action of $G_2$ is included into the fixed point subalgebra by the action of $G_1$. Hence we would like to act by the largest possible subgroup of $U_n$ to make the fixed point subalgebra as small as possible. In the following theorem we prove that the largest subgroup exists for arbitrary choice of parameters $\{ q_{ij} \}$.

\begin{theorem}
    Let $G = \{ u \in U_n \text{ } | \text{ }q_{kl}\overline{u_{ki}} u_{lj} = q_{ij} u_{lj} \overline{u_{ki}}, \text{ }  i, j, k, l = 1,\ldots,n \} $. Then $G$ is a closed subgroup of $U_n$. 
\end{theorem}
\begin{proof}

\quad

1) Let $u$ be the identity matrix. Then \[ q_{kl}\overline{u_{ki}} u_{lj} = q_{kl} \delta_{ki} \delta_{lj} = q_{ij} \delta_{lj} \delta_{ki} = q_{ij} u_{lj} \overline{u_{ki}}. \]

2) Suppose $u \in G$. We prove that $u^{-1} = u^* \in G$:

If $u_{ik}$ or $u_{jl}$ equal to 0 then trivially 
\[ q_{kl}\overline{u^*_{ki}} u^*_{lj} = q_{ij} u^*_{lj} \overline{u^*_{ki}}. \]
Otherwise, for $u \in G$ we know that
\[q_{ij}\overline{u_{ik}} u_{jl} = q_{kl} u_{jl} \overline{u_{ik}}, \]
so $q_{kl} = q_{ij}$ and 
\[ q_{kl}\overline{u^*_{ki}} u^*_{lj} = q_{ij} u^*_{lj} \overline{u^*_{ki}}. \]

3) Suppose $u, v \in G$. Let $h = uv$.

As $u, v \in G$ we get
\[ q_{kl} \overline{u_{k\alpha}} u_{l\beta} = q_{\alpha \beta
} u_{l \beta} \overline{u_{k\alpha}}, \text{ } q_{\alpha \beta
} \overline{v_{\alpha i}} v_{\beta j} = q_{ij} v_{\beta
 j} \overline{v_{\alpha i}}  \]
for $\alpha,\beta = 1,\ldots,n$. Then

\begin{equation*}
\begin{split}
q_{kl} \overline{h_{ki}} h_{lj} = & \sum_{\alpha,\beta} q_{kl} \overline{u_{k\alpha}} \overline{v_{\alpha i}} u_{l\beta} v_{\beta j} = \\ & = \sum_{\alpha,\beta} q_{\alpha \beta} u_{l\beta} \overline{v_{\alpha i}} \overline{u_{k\alpha}} v_{\beta j}  = \\ & = \sum_{\alpha,\beta} q_{ij} u_{l\beta} v_{\beta j} \overline{u_{k\alpha}} \overline{v_{\alpha i}} = q_{ij} h_{lj} \overline{h_{ki}}.
\end{split}
\end{equation*}

\end{proof}

Consider the subgroup of diagonal matrices in $U_n$. It is isomorphic to $\mathbb{T}^n$. This subgroup has the following special properties:
\begin{enumerate}
    \item It defines a correct action on every $q_{ij}$-$\mathcal{CCR}$ regardless of the choice of the parameters. Namely, if $u$ is diagonal then
    \[q_{kl} \overline{u_{ki}} u_{lj} = q_{kl} \delta_{ki} \delta_{lj} \overline{u_{ki}} u_{lj} = q_{ij} \delta_{lj} \delta_{ki} u_{lj} \overline{u_{ki}} = q_{ij} u_{lj} \overline{u_{ki}}. \] 
    \item There exists a special choice of $\{ q_{ij} \}$ such that the group defined in the Theorem 1 is isomorphic to $\mathbb{T}^n$. Indeed, take arbitrary $\{ q_{ij} \}$ with $q_{ii} \neq q_{jj}$ whenever $i \neq j$. Then in particular 
    \[ \overline{u_{ji}} u_{ji} (q_{jj} - q_{ii}) = 0. \]
    Hence $u_{ji} = 0$ for $j \neq i$.
\end{enumerate}

We denote the action of $\mathbb{T}^n$ defined above by $\Psi$. For $\mathbf{w} = (w_1, \ldots, w_n) \in \mathbb{T}^n$ let $\Psi_{\mathbf{w}} : q_{ij}\text{-}\mathcal{CCR} \rightarrow q_{ij}\text{-}\mathcal{CCR}$, \[\Psi_{\mathbf{w}}(a_i) = w_i a_i, \text{ } i = 1,\ldots, n,\]
be the corresponding automorphism.

\section{Fixed point subalgebras and nuclearity}

In this Section we describe the fixed point subalgebra for the action $\Psi$ on $q_{ij}$-$\mathcal{CCR}$ for any values of $q_{ij}$ such that $\max |q_{ij}| < 1$. Then for $Isom_{q_{ij}}$ we prove that the fixed point subalgebra is AF. 

Recall that applying $q_{ij}$-relations, any monomial in $a_i, a_i^*$ can be brought into a normal form with all starred operators to the right of all unstarred ones. For multi-index $\mu = (\mu_1, \ldots, \mu_k)$, $\mu_i \in \{1, \ldots, n \}$ we denote \[ a_\mu = a_{\mu_1} \ldots a_{\mu_k}. \]
The following result follows obviously from the Diamond lemma (see \cite{Diamond_lemma}).
\begin{proposition}
Monomials of the form $a_\mu a_\sigma^*$, where $\mu, \sigma$ vary over all multi-indices, is a linear basis of $q_{ij}$-$CCR$.
\end{proposition}

The following technical statement can be easily get (see \cite{Black})

\begin{proposition}
Let $\gamma$ be an action of a compact group on a $C^*$-algebra $\mathcal{B}$ and $\mathcal{B}_0$ be a dense $*$-subalgebra of $\mathcal{B}$.
If $E_\gamma$ is a conditional expectation onto the corresponding fixed point subalgebra $\mathcal{B}^\gamma$ of $\mathcal{B}$ and $E_\gamma(\mathcal{B}_0) \subset \mathcal{B}_0$, then $\mathcal{B}_0 \cap \mathcal{B}^\gamma$ is dense in $\mathcal{B}^\gamma$.
\end{proposition}

Put $occ_i(\mu)$ to be the number of occurrences of $i$ in a multi-index $\mu$. For a multi-index $\mu$ let $occ(\mu) = (occ_1(\mu), \ldots, occ_n(\mu)) \in \mathbb{Z}_{+}^n$. Write $occ(\mu) = occ(\sigma)$ if $occ_i(\mu) = occ_i(\sigma)$ for any $i = 1, \ldots, n$.

Consider the linear space $GICAR_{q_{ij}} = span\{a_\mu a_\sigma^* \text{ } | \text{ } occ(\mu) = occ(\sigma) \}.$

\begin{theorem}
If $x \in q_{ij}$-$CCR$ then $\Psi_{\mathbf{w}}(x) = x$ for any $\mathbf{w} \in \mathbb{T}^n$ if and only if $x \in GICAR_{q_{ij}}$.
\end{theorem}
\begin{proof}
Compute an action of $\Psi_{\mathbf{w}}$, $\mathbf{w} = (w_1, \ldots, w_n)$ on the basis of $q_{ij}$-$CCR$:
\begin{equation*}
\begin{split}
\Psi_{\mathbf{w}}(a_\mu a_\sigma^*) & =  w_{\mu_1} \ldots w_{\mu_k} \overline{w_{\sigma_1}} \ldots \overline{w_{\sigma_k}} a_\mu a_\sigma^* = \\ & = w_1^{occ_1(\mu)} \overline{w_1}^{occ_1(\sigma)} \ldots w_{n}^{occ_n(\mu)} \overline{w_{n}}^{occ_n(\sigma)} a_\mu a_\sigma^*.
\end{split}
\end{equation*}

If $a_\mu a_\sigma^* \in GICAR_{q_{ij}}$ then for any $i = 1, \ldots, n$ we have $w_i^{occ_i(\mu)} \overline{w_i}^{occ_i(\sigma)} = 1$, so $\Psi_{\mathbf{w}}(a_\mu a_\sigma^*) = a_\mu a_\sigma^*$. Thus if $x \in GICAR_{q_{ij}}$ then $\Psi_{\mathbf{w}}(x) = x$.

Conversely, if $x \in q_{ij}$-$CCR$ then we can write $x = \sum_{\mu, \sigma} C_{\mu, \sigma} a_\mu a_\sigma^*$. Suppose $\Psi_{\mathbf{w}}(x) = x$ for any $\mathbf{w} \in \mathbb{T}^{n}$. In particular it is true for $\mathbf{w}_i = (1, \ldots, 1, z, 1, \ldots, 1)$, where $z$ is on the $i$-th place. Then 
\[ \Psi_{\mathbf{w}_i}(x) - x = \sum_{\mu, \sigma} (z^{occ_i(\mu)} \overline{z}^{occ_i(\sigma)} - 1) C_{\mu, \sigma} a_\mu a_\sigma^* = 0 \text{ for any } z \in \mathbb{T}.\]
Since $a_\mu a_\sigma^*$ are linearly independent, the above equality implies $occ(\mu) = occ(\sigma)$ whenever $C_{\mu, \sigma} \neq 0$.

\end{proof}

Theorem 2 easily implies that $GICAR_{q_{ij}}$ is a $*$-subalgebra of $q_{ij}$-$\mathcal{CCR}$. Combining Proposition 4 with Theorem 2 we obtain the following corollary.

\begin{corollary}
Let $\mathcal{GICAR}_{q_{ij}} = \overline{GICAR}_{q_{ij}}$. Then the fixed point subalgebra $(q_{ij}$-$\mathcal{CCR})^{\Psi}$ coincides with $\mathcal{GICAR}_{q_{ij}}$.
\end{corollary}

Below we study in a more details the structure of $\mathcal{GICAR}_{q_{ij}}$ in the case when $q_{ii} = 0$, $i = 1,\ldots,n$. In the following lemma we present the multiplication rules for elements of $\mathcal{GICAR}_{q_{ij}}$.

\begin{lemma}
Let $a_{\mu(1)} a_{\sigma(1)}^*, a_{\mu(2)} a_{\sigma(2)}^* \in \mathcal{GICAR}_{q_{ij}}$. Then $a_{\mu(1)} a_{\sigma(1)}^* a_{\mu(2)} a_{\sigma(2)}^* = \alpha a_{\mu(3)} a_{\sigma(3)}^*$ for some $\alpha \in \mathbb{C}$. Moreover, \[occ_i(\mu(3)) = \max(occ_i(\mu(1)), occ_i(\mu(2)))\] for $i = 1, \ldots, n$.
\end{lemma}
\begin{proof}
Rewrite the product $a_{\mu(1)} a_{\sigma(1)}^*$ and $a_{\mu(2)} a_{\sigma(2)}^*$ as an element of $span \{a_\mu a_\sigma^* \}$:
\begin{enumerate}
    \item If $\sigma(1)_i \neq \mu(2)_k$ for any $k$, then move $a_{\sigma(1)_i}^*$ to the left of $a_{\mu(2)}$ using $q_{ij}$-commutation relations.
    \item If $\sigma(1)_i = \mu(2)_k$ for some $k$, then move $a_{\sigma(1)_i}^*$ left to $a_{\mu(2)_k}$ using $q_{ij}$-commutation relations and annihilate them using the fact that $a_{\sigma(1)_i}$ is an isometry.
\end{enumerate}

As a result we get an expression of the form $\alpha a_{\mu(3)} a_{\sigma(3)}^*$ for some $\alpha \in \mathbb{C}$.

If $occ_i(\sigma(1)) > occ_i(\mu(2))$, then each occurrence of $a_i$ in $a_{\mu(2)}$ annihilates with the corresponding element of $a_{\sigma(1)}^*$. Hence, $occ_i(\mu(3)) = occ_i(\mu(1))$ and $occ_i(\sigma(3)) = occ_i(\sigma(2)) + (occ_i(\sigma(1)) - occ_i(\mu(2)))$. But $occ_i(\mu(1)) = occ_i(\sigma(1))$ and $occ_i(\mu(2)) = occ_i(\sigma(2))$. Hence $occ_i(\sigma(3)) = occ_i(\mu(3)) = occ_i(\sigma(1))$.
     
If $occ_i(\sigma(1)) \leq occ_i(\mu(2))$, then the same arguments as in the first case lead us to $occ_i(\sigma(3)) = occ_i(\mu(3)) = occ_i(\mu(2))$.
  
Therefore $ occ_i(\mu(3)) = \max(occ_i(\mu(1)), occ_i(\mu(2))). $
\end{proof}

Put 
\begin{equation}
    \mathcal{W}_k = span\{a_\mu a_\sigma^* \in \mathcal{GICAR}_{q_{ij}} \text{ } | \text{ } \max_i occ_i(\mu) \leq k \}. 
\end{equation}
Lemma 2 implies the following statement.
\begin{corollary}
$\mathcal{W}_k$ is a finite-dimensional $*$-subalgebra of $\mathcal{GICAR}_{q_{ij}}$.
\end{corollary}

\begin{theorem}
    $\mathcal{GICAR}_{q_{ij}}$ is an AF algebra.
\end{theorem}
\begin{proof}
Since any $x \in GICAR_{q_{ij}}$ belongs to $\mathcal{W}_k$ for sufficiently large $k$, we get
\[ GICAR_{q_{ij}} = \bigcup_k \mathcal{W}_k. \]
Hence $\mathcal{GICAR}_{q_{ij}}$ is an AF algebra.
\end{proof}

Every AF algebra is nuclear (see \cite{Black}). Therefore we obtain an important corollary of Theorem 3 and Proposition 1.

\begin{theorem}
$Isom_{q_{ij}}$ is nuclear for arbitrary $\{ q_{ij} \}$ such that $\max |q_{ij}| < 1$.
\end{theorem}

\section{Stability of $\mathcal{GICAR}_{q_{ij}}$}

In this section we prove that the structure of $\mathcal{GICAR}_{q_{ij}}$ does not depend on $\{q_{ij} \}$, i.e. $\mathcal{GICAR}_{q_{ij}} \simeq \mathcal{GICAR}_{0}$ for any values of $q_{ij}$, $\max |q_{ij}| < 1$. For this purpose we compute the Bratelli diagram of $\mathcal{GICAR}_{q_{ij}}$.

Denote by $\mathbb{Z}_{+}^n$ a set of all $n$-tuples of non-negative integers. For $\mathbf{v} = (v_1, \ldots, v_n),$  $\mathbf{u} = (u_1, \ldots, u_n) \in \mathbb{Z}_{+}^n$ put $\mathbf{v} + \mathbf{u}$, $\mathbf{v} - \mathbf{u}$, $\max(\mathbf{v}, \mathbf{u})$ for an $n$-tuple which is obtained by componentwise application of the corresponding functions to the entries of $\mathbf{v}$ and $\mathbf{u}$. We write $\mathbf{v} \leq \mathbf{u}$ if $v_i \leq u_i$ and $\mathbf{v} = \mathbf{u}$ if $v_i = u_i$ for all $i = 1, \ldots, n$. For $S \subset \{1,\ldots,n\}$ denote by $\delta_S$ $n$-tuple which has 1 in the $i$-th entry when $i \in S$ and 0 otherwise. For $k \geq 0$ write $\mathbf{k}^n$ for the $n$-tuple $(k,\ldots,k)$. 

For $\mathbf{v} \in \mathbb{Z}_{+}^n$ we put \[ \mathcal{W}_\mathbf{v} = span\{ a_\mu a_\sigma^* \text{ } | \text{ } occ(\mu) = occ(\sigma) = \mathbf{v}\}. \]

From Lemma 2 one can easily get that
\begin{equation}
    \mathcal{W}_{\mathbf{v}} \cdot \mathcal{W}_{\mathbf{u}} \subset \mathcal{W}_{\max(\mathbf{v}, \mathbf{u})}.
\end{equation}

As a consequence of this inclusion, we have that $\mathcal{W}_{\mathbf{v}}$ is closed under multiplication, so it is a $*$-subalgebra of $\mathcal{GICAR}_{q_{ij}}$. In the following two lemmas we construct some faithful $*$-representation of $\mathcal{W}_{\mathbf{v}}$.

\begin{lemma}
For $\mathbf{v} \in \mathbb{Z}_{+}^n$ put
\[ \mathcal{H}_{\mathbf{v}} = span \{ a_\alpha \text{ } | \text{ } occ(\alpha) = \mathbf{v} \}. \]
Then for any $a_{\alpha}, a_{\beta} \in \mathcal{H}_{\mathbf{v}}$ one has $a_{\beta}^* a_{\alpha} = \lambda_{\alpha, \beta} 1$ for some $\lambda_{\alpha, \beta} \in \mathbb{C}$.
Equip $\mathcal{H}_{\mathbf{v}}$ with the Hermitian form $\langle a_\mu, a_\sigma \rangle_{\mathbf{v}} = \lambda_{\mu, \sigma}$.
Then $\langle \cdot, \cdot \rangle_{\mathbf{v}}$ is positive-definite.

\end{lemma}
\begin{proof}
To prove this lemma we show that $\langle a_\mu, a_\sigma \rangle_{\mathbf{v}}$ coincides with the restriction of the Fock inner product to the subspace spanned by the tensors of the following type
\[ \{ \xi_{\alpha} = \xi_{\alpha_1} \otimes \ldots \otimes \xi_{\alpha_m} \text{ } | \text{ } \alpha = (\alpha_1, \ldots, \alpha_m), \text{ } occ(\alpha) = \mathbf{v} \}, \]
where $m = v_1 + \ldots + v_n$. Here we identify naturally $a_\alpha$ with $\xi_{\alpha_1} \otimes \ldots \otimes \xi_{\alpha_m}$.
We prove the statement by induction on $v_1 + \ldots + v_n$. Then we use the fact that the Fock inner product is positive definite, see \cite{Bozejko_brownian_motion}.

If $\mathbf{v} = 0$ then \[ \langle 1, 1 \rangle_{\mathbf{0}} = 1 = \langle \Omega, \Omega \rangle_{Fock}. \]

Assume $v_1 + \ldots + v_n = m + 1$ for some $m \geq 0$. Take $a_\mu, a_\sigma \in \mathcal{H}_{\mathbf{v}}$. Let $t_0$ be index of the first occurrence of $\sigma_1$ in $\mu$. Since $q_{ii} = 0$ for $i = 1, \ldots, n$, we get
\begin{gather*}
    \langle \xi_{\mu_1} \otimes \ldots \otimes \xi_{\mu_{m+1}}, \xi_{\sigma_1} \otimes \ldots \otimes \xi_{\sigma_{m+1}} \rangle_{Fock} = \\ = \sum_{t = 1}^{m+1} \delta_{\sigma_1 \mu_t} q_{\sigma_1 \mu_1} \ldots q_{\sigma_1 \mu_{t - 1}} \langle \xi_{\mu_1} \otimes \ldots \otimes \widehat{\xi_{\mu_t}} \otimes \ldots \otimes \xi_{\mu_{m+1}}, \xi_{\sigma_2} \otimes \ldots \otimes \xi_{\sigma_{m+1}} \rangle_{Fock} = \\ =
    q_{\sigma_1 \mu_1} \ldots q_{\sigma_1, \mu_{t_0 - 1}} \langle \xi_{\mu_1} \otimes \ldots \otimes \widehat{\xi_{\mu_{t_0}}} \otimes \ldots \otimes \xi_{\mu_{m+1}}, \xi_{\sigma_2} \otimes \ldots \otimes \xi_{\sigma_{m+1}} \rangle_{Fock} = \\ =
    q_{\sigma_1 \mu_1} \ldots q_{\sigma_1, \mu_{t_0 - 1}} \langle a_{\mu_1} \ldots a_{\mu_{t_0 - 1}} a_{\mu_{t_0 + 1}} \ldots a_{\mu_{m + 1}}, a_{\sigma_2} \ldots a_{\sigma_{m+1}} \rangle_{\mathbf{v} - \mathbf{\delta}_{\{\sigma_1\}}} = \\ =
    q_{\sigma_1 \mu_1} \ldots q_{\sigma_1, \mu_{t_0 - 1}} \langle a_{\mu_1} \ldots a_{\mu_{t_0 - 1}} a_{\sigma_1}^* a_{\mu_{t_0}} a_{\mu_{t_0 + 1}} \ldots a_{\mu_{m + 1}}, a_{\sigma_2} \ldots a_{\sigma_{m+1}} \rangle_{\mathbf{v} - \mathbf{\delta}_{\{\sigma_1\}}} = \\ =
    \langle a_{\sigma_1}^* a_{\mu_1} \ldots a_{\mu_{m + 1}}, a_{\sigma_2} \ldots a_{\sigma_{m+1}} \rangle_{\mathbf{v} - \mathbf{\delta}_{\{\sigma_1\}}} = \langle a_\mu, a_\sigma \rangle_{\mathbf{v}}.
\end{gather*}
\end{proof}

In what follows we write $M_n$ for the $C^*$-algebra of $n$ by $n$ complex matrices.

\begin{lemma}
For $\mathbf{v} \in \mathbb{Z}_{+}^n$ define a $*$-representation $\pi_{\mathbf{v}} : \mathcal{W}_{\mathbf{v}} \rightarrow \mathcal{B}(\mathcal{H}_{\mathbf{v}})$:
    \[ \pi_{\mathbf{v}}(a_\mu a_\sigma^*)(a_\alpha) = a_\mu a_\sigma^* a_\alpha =  \langle a_\alpha, a_\sigma \rangle_{\mathbf{v}} a_\mu. \]
    
Then $\pi_{\mathbf{v}}$ is a faithful and surjective $*$-representation of $\mathcal{W}_{\mathbf{v}}$ and

\[\mathcal{W}_{\mathbf{v}} \simeq \mathcal{B}(\mathcal{H}_{\mathbf{v}}) \simeq M_{\frac{(v_1 + \ldots + v_n)!}{v_1!\ldots v_n!}}. \]
\end{lemma}
\begin{proof}
Recall that the family of all rank one operators on a finite-dimensional Hilbert space $\mathcal{H}$ generates the whole algebra $\mathcal{B}(\mathcal{H})$. Hence surjectivity of $\pi_{\mathbf{v}}$ follows from the fact that every rank one operator on $\mathcal{H}_{\mathbf{v}}$ is in the image of $\pi_{\mathbf{v}}$.

Observe that $\dim \mathcal{W}_{\mathbf{v}} = (\frac{(v_1 + \ldots + v_n)!}{v_1!\ldots v_n!})^2 = \dim \mathcal{B}(\mathcal{H}_{\mathbf{v}})$. Therefore $\pi_{\mathbf{v}}$ is injective, since it is surjective.
\end{proof}

Notice that $\mathcal{W}_{\mathbf{v}}$ and $\mathcal{W}_{\mathbf{u}}$ are not orthogonal with respect to multiplication for $\mathbf{v} \neq \mathbf{u}$. Hence we can not simply take direct sum of $\pi_{\mathbf{v}}$ to obtain a $*$-representation of the subalgebras $\mathcal{W}_k$ defined by (1). Nevertheless the following result holds

\begin{theorem}
    \begin{equation*}
        \mathcal{W}_k \simeq \bigoplus_{\mathbf{v} \leq \mathbf{k}^n} M_{\frac{(v_1 + \ldots + v_n)!}{v_1!\ldots v_n!}}
    \end{equation*}
\end{theorem}
\begin{proof}

We build $\mathcal{W}_k$ inductively beginning from multi-indices $\mathbf{v}$ with the largest value of $\sum_t \mathbf{v}_t$.

1) Put $\mathcal{W}_k^{(1)} = \mathcal{W}_k$. Then $\mathcal{W}_{\mathbf{k}^n} \subset \mathcal{W}_k^{(1)}$. By property (2) it is an ideal in $\mathcal{W}_k^{(1)}$. Lemma 4 implies that $\mathcal{W}_{\mathbf{k}^n} \simeq M_{\frac{(nk)!}{(k!)^n}}$. Let $\mathcal{W}_k^{(2)} = \mathcal{W}_k^{(1)} / \mathcal{W}_{\mathbf{k}^n}$. Then due to $\dim \mathcal{W}_k < \infty$ one has
\[\mathcal{W}_k \simeq M_{\frac{(nk)!}{(k!)^n}} \oplus \mathcal{W}_k^{(2)}. \]

2) Assume by induction that $\mathcal{W}_k \simeq J \oplus \mathcal{W}_k^{(j)}$ where \[ J = \bigcup_{v_1 + \ldots + v_n > nk - j + 1} \mathcal{W}_{\mathbf{v}} \simeq \bigoplus_{v_1 + \ldots + v_n > nk - j + 1} M_{\frac{(v_1 + \ldots + v_n)!}{v_1!\ldots v_n!}} \]
and $\mathcal{W}_k^{(j)} = \mathcal{W}_k / J$. Let $\sigma_j : \mathcal{W}_k \rightarrow \mathcal{W}_k^{(j)}$ be a projection. Let $\mathcal{W}_{\mathbf{v}}$ and $\mathcal{W}_{\mathbf{u}}$ in $\mathcal{W}_k$ be such that $\sum_t v_t = \sum_t u_t = nk - j + 1$. By property (2) if $\mathbf{v} \neq \mathbf{u}$ then $\mathcal{W}_{\mathbf{v}} \cdot \mathcal{W}_{\mathbf{u}} \subset \mathcal{W}_{\max(\mathbf{v},\mathbf{u})} \subset J$. Hence $\sigma_j(\mathcal{W}_{\mathbf{v}})$ and $\sigma_j(\mathcal{W}_{\mathbf{u}})$
are ideals in $\mathcal{W}_k^{(j)}$ such that $\sigma_j(\mathcal{W}_{\mathbf{v}}) \cdot \sigma_j(\mathcal{W}_{\mathbf{u}}) = 0$.  By Lemma 4 \[ \sigma_j(\mathcal{W}_{\mathbf{v}}) \simeq \mathcal{W}_{\mathbf{v}} \simeq M_{\frac{(v_1 + \ldots + v_n)!}{v_1!\ldots v_n!}}.\] Let \[ 
\mathcal{W}_k^{(j + 1)} = \mathcal{W}_k^{(j)} / \bigoplus_{v_1 + \ldots + v_n = nk - j + 1} \sigma_j(\mathcal{W}_{\mathbf{v}}).\] Then
\[\mathcal{W}_k^{(j)} = \bigoplus_{v_1 + \ldots + v_n = nk - j + 1} M_{\frac{(v_1 + \ldots + v_n)!}{v_1!\ldots v_n!}} \oplus \mathcal{W}_k^{(j+1)}. \]

3) To complete the proof it remains to note that with $j = nk$ we get $\mathcal{W}_k^{(nk)} \simeq \mathbb{C}$.
    
\end{proof}

Denote by $\mathcal{V}_{\mathbf{v}}^k$ the component $M_{\frac{(v_1 + \ldots + v_n)!}{v_1!\ldots v_n!}}$ of $\mathcal{W}_k$ from the decomposition above. 

\begin{remark}
Let $1_{\mathbf{v}}^k$ be a unit of $\mathcal{V}_{\mathbf{v}}^k$. Then the proof of Theorem 5 implies that $\mathcal{V}_{\mathbf{v}}^k$ coincides with $\mathcal{W}_{\mathbf{v}} \cdot 1_{\mathbf{v}}^k$. Hence,
   \begin{equation*}
        \mathcal{W}_k = \bigoplus_{\mathbf{v} \leq \mathbf{k}^n} \mathcal{V}_{\mathbf{v}}^k =  \bigoplus_{\mathbf{v} \leq \mathbf{k}^n} \mathcal{W}_{\mathbf{v}} \cdot 1_{\mathbf{v}}^k.
    \end{equation*}
\end{remark}

The structure of $\mathcal{W}_k$ is independent of the values of $q_{ij}$ as it follows from Theorem 5. For the analysis of embedding of $\mathcal{W}_k$ into $\mathcal{W}_{k + 1}$ we construct a special $*$-representation of $\mathcal{W}_k$. 

In Lemma 4 we have constructed a bijective $*$-representation $\pi_{\mathbf{v}}$ of $\mathcal{W}_{\mathbf{v}}$. Since $\mathcal{V}^k_{\mathbf{v}} =  \mathcal{W}_{\mathbf{v}} \cdot 1_{\mathbf{v}}^k$, we can define a bijective $*$-representation $\widehat{\pi}^k_{\mathbf{v}}$ \[ \widehat{\pi}^k_{\mathbf{v}} : \mathcal{V}^k_{\mathbf{v}} \rightarrow \mathcal{B}(\mathcal{H}_{\mathbf{v}}), \text{ } \widehat{\pi}^k_{\mathbf{v}}(x 1_{\mathbf{v}}^k) = \pi_{\mathbf{v}}(x), \text{ } x \in \mathcal{W}_{\mathbf{v}}. \]
Put \[ \mathcal{H}_k = \bigoplus_{\mathbf{v} \leq \mathbf{k}^n} \mathcal{H}_{\mathbf{v}}.  \]
Let $x \in \mathcal{W}_k$. Then there exists $y_{\mathbf{v}} \in \mathcal{W}_{\mathbf{v}}$ such that $x 1_{\mathbf{v}}^k = y_{\mathbf{v}} 1_{\mathbf{v}}^k$. Define the $*$-representation $\pi_k : \mathcal{W}_k \rightarrow \mathcal{B}(\mathcal{H}_k)$ by
\[ \pi_k(x) = \bigoplus_{\mathbf{v} \leq \mathbf{k}^n} \widehat{\pi}^k_{\mathbf{v}}(x 1_{\mathbf{v}}^k) = \bigoplus_{\mathbf{v} \leq \mathbf{k}^n} \widehat{\pi}^k_{\mathbf{v}}(y_{\mathbf{v}} 1_{\mathbf{v}}^k) = \bigoplus_{\mathbf{v} \leq \mathbf{k}^n} \pi_{\mathbf{v}}(y_{\mathbf{v}}). \]
$\pi_k$ is obviously faithful since if $x \in \ker \pi_k$ then $\widehat{\pi}^k_{\mathbf{v}}(x 1_{\mathbf{v}}^k) = 0$. Since $\widehat{\pi}^k_{\mathbf{v}}$ is faithful, $x 1_{\mathbf{v}}^k = 0$ for any $\mathbf{v} \leq \mathbf{k}^n$. 

\begin{lemma}
Let $\mathbf{u}, \mathbf{v} \in \mathbb{Z}_{+}^n$, $\mathbf{u}, \mathbf{v} \leq \mathbf{k}^n$. If $\mathbf{v} \nleq \mathbf{u}$ then $\pi_k(\mathcal{W}_{\mathbf{v}})|_{\mathcal{H}_{\mathbf{u}}} = 0$.
\end{lemma}
\begin{proof}
    Let \[ J = \bigcup_{\mathbf{v} \leq \mathbf{j} \leq \mathbf{k}^n} \mathcal{W}_{\mathbf{j}}.  \]
    $J$ is an ideal in $\mathcal{W}_k$ by property (1). If $\mathbf{v} \nleq \mathbf{u}$ then $J$ has the zero intersection with $\mathcal{V}_{\mathbf{u}}^k$. However $\mathcal{W}_{\mathbf{v}} \cdot 1_{\mathbf{u}}^k \subset J \cap \mathcal{V}_{\mathbf{u}}^k$, so $\mathcal{W}_{\mathbf{v}} \cdot 1_{\mathbf{u}}^k = 0$ and $\pi_k(\mathcal{W}_{\mathbf{v}})|_{\mathcal{H}_{\mathbf{u}}} = \widehat{\pi}_{\mathbf{u}}^k(\mathcal{W}_{\mathbf{v}} \cdot 1_{\mathbf{u}}^k) = 0$.
    
\end{proof}

Let $\mathbf{u} \geq \mathbf{v}$ and $P_{\mathbf{v}}^{\mathbf{u}} \in \mathcal{B}(\mathcal{H}_{\mathbf{u}})$ be the orthogonal projection onto the subspace of $\mathcal{H}_{\mathbf{u}}$ which is spanned by monomials $a_{\mu} \in \mathcal{H}_{\mathbf{u}}$ such that $occ(\mu_1,\ldots,\mu_{|\mathbf{v}|}) = \mathbf{v}$. Notice that this subspace has dimension equal to $\dim \mathcal{H}_{\mathbf{v}} \cdot \dim \mathcal{H}_{\mathbf{u}-\mathbf{v}}$. Put  $p_{\mathbf{v}}^{\mathbf{u}} = \pi^{-1}_{\mathbf{u}}(P_{\mathbf{v}}^{\mathbf{u}}) \in \mathcal{W}_{\mathbf{u}}$. If $\mathbf{u} \nleq \mathbf{w}$ then $\pi_k(p_{\mathbf{v}}^{\mathbf{u}})|_{\mathcal{H}_{\mathbf{w}}} = 0$ due to Lemma 5. Also observe that if $\mathbf{u_1} \leq \mathbf{w}$ and $\mathbf{u_2} \leq \mathbf{w}$ then $\pi_k(p_{\mathbf{v}}^{\mathbf{u_1}})|_{\mathcal{H}_{\mathbf{w}}} = \pi_k(p_{\mathbf{v}}^{\mathbf{u_2}})|_{\mathcal{H}_{\mathbf{w}}}$.

In the following lemma we express $1_{\mathbf{v}}^k$ as a sum of the projections $p_{\mathbf{v}}^{\mathbf{u}}$.
\begin{lemma}
Let $\mathbf{v} \in \mathbb{Z}_{+}^n$ such that $\mathbf{v} \leq \mathbf{k}^n$. Then
\[ 1_{\mathbf{v}}^k = \sum_{\substack{S \subset \{1,\ldots,n\} \\ \mathbf{v} + \delta_S \leq \mathbf{k}^n}} (-1)^{l(\delta_S)} p^{\mathbf{v}+\delta_S}_{\mathbf{v}}, \]
where $l(\delta_S) = \sum_{i=1}^n (\delta_S)_i$.
\end{lemma}
\begin{proof}
We show that the defined sum acts as the identity operator on $\mathcal{H}_{\mathbf{v}}$ and as the zero operator on $\mathcal{H}_{\mathbf{u}}$ for $\mathbf{u} \neq \mathbf{v}$ in the representation $\pi_k$. Consider the following cases:

1) If $\mathbf{v} = \mathbf{u}$ then whenever $\delta_S \neq \mathbf{0}^n$, we have \[ \pi_k(p^{\mathbf{v} + \delta_S}_{\mathbf{v}})|_{\mathcal{H}_{\mathbf{u}}} = 0.\] Hence \[ \sum_{\substack{S \subset \{1,\ldots,n\} \\ \mathbf{v} + \delta_S \leq \mathbf{k}^n}} (-1)^{l(\delta_S)} \pi_k(p^{\mathbf{v}+\delta_S}_{\mathbf{v}})|_{\mathcal{H}_{\mathbf{u}}} = \pi_k(p_{\mathbf{v}}^{\mathbf{v}})|_{\mathcal{H}_{\mathbf{u}}} = id_{\mathcal{H}_\mathbf{u}} .\]

2) If $\mathbf{v} \nleq \mathbf{u}$ then by Lemma 5 $\pi_k(p^{\mathbf{v}+\delta_S}_{\mathbf{v}})|_{\mathcal{H}_{\mathbf{u}}} = 0$ for arbitrary $S \subset \{1,\ldots,n\}$.

3) In other cases $\mathbf{v} < \mathbf{u} \leq \mathbf{k}^n$. If $\mathbf{u}_t > \mathbf{v}_t$ then for $S \subset \{1,\ldots,n\}$ one has $\mathbf{v}+\delta_S \leq \mathbf{u}$ if and only if $\mathbf{v}+\delta_{S \setminus \{t\}} \leq \mathbf{u}$, so
\[ \pi_k(p^{\mathbf{v}+\delta_S}_{\mathbf{v}})|_{\mathcal{H}_{\mathbf{u}}} = \pi_k(p^{\mathbf{v}+\delta_{S \setminus \{t\}}}_{\mathbf{v}})|_{\mathcal{H}_{\mathbf{u}}}. \]
Hence
\begin{equation*}
\begin{split}
   & \sum_{\substack{S \subset \{1,\ldots,n\} \\ \mathbf{v} + \delta_S \leq \mathbf{k}^n}} (-1)^{l(\delta_S)} \pi_k(p^{\mathbf{v}+\delta_S}_{\mathbf{v}})|_{\mathcal{H}_{\mathbf{u}}} = \\ & =
      \sum_{\substack{S \subset \{1,\ldots,n\} \setminus \{t\} \\ \mathbf{v} + \delta_S \leq \mathbf{k}^n}} ( (-1)^{ l(\delta_{S \setminus \{t\}}) } \pi_k(p^{\mathbf{v}+\delta_S}_{\mathbf{v}})|_{\mathcal{H}_{\mathbf{u}}} + (-1)^{ 1+ l(\delta_{S \setminus \{t\}}) } \pi_k(p_{\mathbf{v}}^{\mathbf{v}+\delta_S+\delta_{\{t\}}})|_{\mathcal{H}_{\mathbf{u}}}) = 0.
        \\
\end{split}
\end{equation*}
\end{proof}

Remind that every homomorphism $\phi : \mathcal{B}(V_1) \rightarrow \mathcal{B}(V_2)$ with finite-dimensional $V_1$ and $V_2$ is determined up to a unitary equivalence by a natural number $m$ and has the form \[ \phi(x) = \left[ \begin{array}{cc}
1_m \otimes x & 0 \\
0 & 0 \end{array} \right]. \] This number can be determined as \[ m = \frac{\dim \text{Im } \phi(id_{V_1}) }{\dim V_1}. \]

\begin{theorem}
    Let $\mathbf{v}, \mathbf{u} \in \mathbb{Z}_{+}^n$ be such that $\mathbf{v} \leq \mathbf{k}^n$, $\mathbf{u} \leq \mathbf{(k + 1)}^n$. Then $\mathcal{V}_{\mathbf{v}}^k$ is embedded into $\mathcal{V}_{\mathbf{u}}^{k+1}$ with nonzero multiplicity $m_{\mathbf{v},\mathbf{u}}$ if and only if $\mathbf{0}^n \leq \mathbf{u} - \mathbf{v} \leq \mathbf{1}^n$ with $u_t > v_t$ only in the case when $v_t = k$. If the multiplicity is nonzero then it is equal to $\frac{(\sum_t u_t - v_t)!}{\prod_t (u_t - v_t)!}$.
\end{theorem}
\begin{proof}

Suppose $\mathbf{v} \nleq \mathbf{u}$. Then by Lemma 5 for $S \subset \{1,\ldots,n\}$ one has $\pi_{k+1}(p^{\mathbf{v}+\delta_S}_{\mathbf{v}})|_{\mathcal{H}_{\mathbf{u}}}$, so
\[ \pi_{k+1}(1_{\mathbf{v}}^k) = \sum_{\substack{S \subset \{1,\ldots,n\} \\ \mathbf{v} + \delta_S \leq \mathbf{k}^n}} (-1)^{l(\delta_S)} \pi_{k+1}(p^{\mathbf{v}+\delta_S}_{\mathbf{v}})|_{\mathcal{H}_{\mathbf{u}}} = 0 .\]
Hence for the embedding to be nonzero we need $\mathbf{v} \leq \mathbf{u}$.

If there exists $t$ such that $\min(u_t,k) > \mathbf{v}_t$ then using the same argument as in the case 3 of Lemma 6 one has $\pi_{k + 1}(1_{\mathbf{v}}^k)|_{\mathcal{H}_{\mathbf{u}}} = 0$. Hence for the embedding to be nonzero we need to have $\mathbf{u} - \mathbf{v} \leq \mathbf{1}^n$ and $u_t = k + 1$ whenever $u_t > v_t$.

Hence for $t = 1,\ldots,n$ we have either $u_t = v_t$ or $u_t = k + 1$, so if $S \subset \{1,\ldots,n\}$ is non-empty then $\mathbf{v} + \delta_S \nleq \min(\mathbf{u},\mathbf{k}^n)$ and $\pi_{k + 1}(1_{\mathbf{v}}^k)|_{\mathcal{H}_{\mathbf{u}}} = \pi_{k + 1}(p_{\mathbf{v}}^{\mathbf{v}})|_{\mathcal{H}_{\mathbf{u}}}$.

Now we calculate the multiplicity $m_{\mathbf{v},\mathbf{u}}$ when the embedding is nonzero:

\begin{equation*}
\begin{split}
m_{\mathbf{v},\mathbf{u}} = & \frac{\dim \{\pi_{k+1}(1_{\mathbf{v}}^k)(\mathcal{H}_{\mathbf{u}}) \} }{\dim \mathcal{H}_{\mathbf{v}}} = \frac{\dim \{\pi_{k+1}(p_{\mathbf{v}}^{\mathbf{v}})(\mathcal{H}_{\mathbf{u}}) \} }{\dim \mathcal{H}_{\mathbf{v}}} = \\& \frac{\dim \{\pi_{k+1}(p_{\mathbf{v}}^{\mathbf{u}})(\mathcal{H}_{\mathbf{u}}) \} }{\dim \mathcal{H}_{\mathbf{v}}} =
\frac{\dim \mathcal{H}_{\mathbf{v}} \dim \mathcal{H}_{\mathbf{u}-\mathbf{v}} }{\dim \mathcal{H}_{\mathbf{v}}} = \frac{(\sum_t \mathbf{u}_t - \mathbf{v}_t)!}{\prod_t (\mathbf{u}_t - \mathbf{v}_t)!}. \\
\end{split}
\end{equation*}

\end{proof}

Hence we proved that the multiplicities of the embeddings do not depend on $q_{ij}$. Since two AF algebras are isomorphic if they have equal Bratelli diagrams, we have $\mathcal{GICAR}_{q_{ij}} \simeq \mathcal{GICAR}_0$.

\section{Faithfulness of the Fock representation}
In this Section we prove that the Fock representation of $Isom_{q_{ij}}$ is faithful. To do this we apply the second part of Proposition 1. In the following lemma we describe an automorphism which interwines $\pi_F$ with the action $\Psi$ of $\mathbb{T}^n$ defined in Section 3.
\begin{lemma}
For every $\mathbf{w} \in \mathbb{T}^n$ there exists an automorphism $\Psi^{Fock}_{\mathbf{w}}$ of $\pi_F(Isom_{q_{ij}})$ such that $\Psi^{Fock}_{\mathbf{w}} \circ \pi_F = \pi_F \circ \Psi_{\mathbf{w}}$.
\end{lemma}
\begin{proof}
Let $A_i' = \mathbf{w}_i \pi_F(a_i)$ and $\Omega$ be a cyclic vector such that $\pi_F(a_i^*)(\Omega) = 0$, $i = 1,\ldots,n$. Then $A_i'$ satisfy $q_{ij}$-commutation relations and $\Omega$ is a cyclic vector such that $A_i'^*(\Omega) = 0$, $i = 1,\ldots,n$. By the uniqueness of the Fock representation, there is a unitary $U$ acting on the deformed Fock space which implements an isomorphism between $\pi_F(Isom_{q_{ij}})$ and the $C^*$-algebra generated by $A_i'$, i.e. \[U^* \pi_F(a_i) U = A_i', \text{ } i = 1,\ldots,n.\] These algebras coincide because $\pi_F(a_i) = \overline{\mathbf{w}_i}a_i'$, so $U$ implements an automorphism of $\pi_F(Isom_{q_{ij}})$. Let $\Psi^{Fock}_{\mathbf{w}}(x) = U^* x U$. Then \[ \Psi^{Fock}_{\mathbf{w}}(\pi_F(a_i)) = U^* \pi_F(a_i) U =A_i' = \mathbf{w}_i \pi_F(a_i) = \pi_F(\mathbf{w}_i a_i) = \pi_F(\Psi_{\mathbf{w}}(a_i)). \] Since the relation $\Psi^{Fock}_{\mathbf{w}} \circ \pi_F = \pi_F \circ \Psi_{\mathbf{w}}$ holds on $\pi_F(a_i)$ and $\{\pi_F(a_i) \text{ } | \text{ } i = 1,\ldots, n \}$ generates $\pi_F(Isom_{q_{ij}})$, the proof is completed.
\end{proof}

Now Proposition 1 can be applied to the action $\Psi$ and the representation $\pi_F$.

\begin{theorem}
The Fock representation $\pi_F$ of the $C^*$-algebra $Isom_{q_{ij}}$ is faithful.
\end{theorem}
\begin{proof}

If $x \in \bigcup_{k} \mathcal{W}_k$ then $x \in\mathcal{W}_k$ for a sufficiently large $k$. Since $1 \in \mathcal{W}_k$, $\norm{x} = \norm{x}_{\mathcal{W}_k}$. It follows from the main result of \cite{Proskur_faithfulness} that $\pi_F$ is faithful on the dense $*-$subalgebra of $Isom_{q_{ij}}$ generated by $a_i, a_i^*$, $i = 1,\ldots,n$. Hence, $\pi_F \vert_{\mathcal{W}_k}$ is an injective $*$-homomorphism from the $C^*$-algebras $\mathcal{W}_k$ to $\pi_F(Isom_{q_{ij}})$, so $\norm{x} = \norm{x}_{\mathcal{W}_k} = \norm{\pi_F(x)}$. Hence by Proposition 1 $\pi_F$ is faithful on $Isom_{q_{ij}}$.
\end{proof}

\section{Description of the ideal $\mathcal{K}$ in $Isom_{q_{ij}}$}
In this Section we describe an ideal $\mathcal{K}$ in $Isom_{q_{ij}}$ which is isomorphic to the algebra of compact operators.

Recall that $\mathcal{K}$ can be described as a universal $C^*$-algebra generated by $e_{ij}$, $i,j \geq 0$ satisfying the relations $e_{ij} e_{kl} = \delta_{jk}e_{il}$, $e_{ij}^* = e_{ji}$ (see \cite{Black}).

When $q_{ij} = 0$ for $i, j = 1,\ldots,n$, we have $Isom_{0} \simeq \mathcal{KO}_n$, which is an extension of $\mathcal{O}_n$ by an ideal isomorphic to $\mathcal{K}$. In this case it is generated by $p = 1 - \sum_k a_k a_k^*$. Notice that $p$ is a nontrivial projection such that \[ p s_i = 0, \text{ } i = 1,\ldots,n.\] In the next theorem we prove that the same conditions are sufficient for an element $p$ in $Isom_{q_{ij}}$ to generate an ideal isomorphic to $\mathcal{K}$.

\begin{theorem}
    Let $p \in Isom_{q_{ij}}$ be a nontrivial projection such that  \[ p s_i = 0, \text{ } i = 1,\ldots,n.\]
    
    Then the ideal generated by $p$ is isomorphic to $\mathcal{K}$.
\end{theorem}
\begin{proof}
By Proposition 3 the span of all words of the form $a_{\mu_1} a_{\sigma_1}^* p a_{\mu_2} a_{\sigma_2}^*$ is dense in the ideal generated by $p$. If $\sigma_1 \neq 0$ or $\mu_2 \neq 0$ then this monomial is equal to $0$ since $a_i^* p = 0$, $p a_i = 0$, $i = 1,\ldots,n$. Therefore the span of monomials of the form $a_\mu p a_\sigma^*$ is dense in the ideal generated by $p$.

Let $\mathcal{H} = span\{a_\mu \text{ } | \text{ } \mu \in \mathbb{Z}_{+}^n \}$. Split $\mathcal{H}$ into the subspaces $\mathcal{H}_{\mathbf{v}}$ and equip it with $\langle a_\mu, a_\sigma \rangle_{\mathbf{v}}$ as in Lemma 3. Apply the orthogonalization process to the basis $\{ a_\mu \}$ of $\mathcal{H}_{\mathbf{v}}$ and denote the result by $\{ \widehat{a_\mu} \}$. Consider the following cases:

\begin{enumerate}
    \item If $occ(\alpha) = occ(\beta) = \mathbf{v} \in \mathbb{Z}_{+}^n$ then $p \widehat{a_\beta}^* \widehat{a_\alpha} p = p \langle \widehat{a_\alpha}, \widehat{a_\beta} \rangle_{\mathbf{v}} p = \delta_{\alpha \beta} p$.
    \item If $occ(\alpha) \neq occ(\beta)$ then $\widehat{a_\beta}^* \widehat{a_\alpha}$ is a non-trivial monomial, so $p \widehat{a_\beta}^* \widehat{a_\alpha} p = 0$.
\end{enumerate}

Put $e_{\alpha \beta} = \widehat{a_\alpha} p \widehat{a_\beta}^*$. Then
\[e_{\alpha \beta} e_{\sigma \mu} = a_\alpha p a_\beta^* a_\sigma p a_\mu^* = \begin{cases} a_\alpha p \delta_{\sigma \beta} a_\mu^* & \mbox{ if } occ(\beta) = occ(\sigma) \\ 0, & \mbox{otherwise} \end{cases} = \delta_{\sigma \beta} a_\alpha p a_\mu^* = \delta_{\beta \sigma} e_{\alpha \mu} \]
    
Hence the ideal generated by $p$ is isomorphic to $\mathcal{K}$.

\end{proof}

So it remains to prove the existence of $p \in Isom_{q_{ij}}$ satisfying the conditions of Theorem 8.

\begin{proposition}
There exists an element $p \in Isom_{q_{ij}}$ such that 
\[ p = p^*, \text{ } p^2 = p, \text{ } pa_i = 0, \text{ } i = 1,\ldots, n. \]
\end{proposition}
\begin{proof}
Let $\mathcal{B}$ be the $C^*$-subalgebra generated by $a_i a_i^*$, $i = 1,\ldots,n$. Since $\mathcal{B} \subset \mathcal{W}_1$, it is finite-dimensional. Every finite-dimensional $C^*$-algebra is unital. Obviously, the unit of $Isom_{q_{ij}}$ does not belong to $\mathcal{B}$. Let $1_{\mathcal{B}}$ be the unit of $\mathcal{B}$. Then for any $i = 1,\ldots,n$ we set
\[(1 - 1_{\mathcal{B}}) a_i = a_i - 1_{\mathcal{B}}a_i = a_i - 1_{\mathcal{B}}a_i (a_i^* a_i) = a_i - (1_{\mathcal{B}}a_i a_i^*) a_i = a_i - (a_i a_i^*) a_i = 0. \]

Thus $p = 1 - 1_{\mathcal{B}}$ is a nontrivial projection satisfying all required conditions.

\end{proof}

\section*{Acknowledgements}
We would like to express our gratitude to Daniil Proskurin, who introduced us to the subject of Wick algebras. His patient guidance was very important for us to make the first steps in mathematics.

Also we thank Lyudmila Turowska and Vasyl Ostrovskyi for reading and giving valuable comments.

\medskip
\printbibliography

\end{document}